\documentclass{article}
\usepackage{graphicx} 
\usepackage{amsmath}
\usepackage{amsfonts}
\usepackage{amssymb}
\usepackage{tikz-cd}
\usepackage{amsthm}
\usepackage{mathrsfs}
\usepackage{changepage}
\usepackage{setspace}
\usepackage{fancyhdr}

\DeclareMathOperator{\Ind}{Ind}

\DeclareMathOperator{\dcl}{dcl}

\def\Ind#1#2{#1\setbox0=\hbox{$#1x$}\kern\wd0\hbox to 0pt{\hss$#1\mid$\hss}
\lower.9\ht0\hbox to 0pt{\hss$#1\smile$\hss}\kern\wd0}

\def\notind#1#2{#1\setbox0=\hbox{$#1x$}\kern\wd0
\hbox to 0pt{\mathchardef\nn=12854\hss$#1\nn$\kern1.4\wd0\hss}
\hbox to 0pt{\hss$#1\mid$\hss}\lower.9\ht0 \hbox to 0pt{\hss$#1\smile$\hss}\kern\wd0}

\newtheorem{theorem}{Theorem}[section]
\newtheorem{lemma}[theorem]{Lemma}

\theoremstyle{remark}

\theoremstyle{definition}
\newtheorem{definition}[theorem]{Definition}

\theoremstyle{definition}

\theoremstyle{definition}
\newtheorem*{exercise*}{Exercise}

\theoremstyle{definition}

\title{Some More Rigid Real Closed Fields}
\author{Michael Lange}
\date{August 2026}

\begin{document}

\onehalfspacing
\setlength{\headheight}{14pt}   
\setlength{\headsep}{8pt} 
\newpage
\maketitle

\maketitle

\section{Intro and Notation}
In [1], the authors prove the existence of countable non-Archimedean real closed fields which are rigid, i.e. have no non-trivial automorphisms. More specifically, they construct such a field which has transcendence degree 2 over the real algebraic numbers. In the following, we show how to extend their construction to produce such extensions for all finite transcendence degrees as well as countably infinite transcendence degree over the real algebraics.
\\
\\
RCF denotes the theory of real closed fields in the language of ordered fields. We let $k$ denote the field of real algebraic numbers, which is the prime model of RCF. Throughout, ``definable" will always mean ``$k$-definable" (equivalently, $\emptyset$-definable). We also fix $R$, a large saturated model of RCF which will contain all the extensions we consider. We will use several standard facts about o-minimal expansions of RCF, such as cell decomposition for definable sets, and piecewise monotonicity, continuity, and differentiability for definable functions. As reference for such facts, see [2]. On $R^n$, the map $\pi_i$ will denote projection onto the $i$th coordinate and $\pi_{\leq i}$ projection onto the first $i$ coordinates. We will only have to deal with open (full dimension) cells, so we mention some abbreviated notation for describing these. In general, an open cell definable over $k$ may be given as:
\[(\alpha,\beta)\times (g_2,h_2)\times\dots\times(g_n,h_n)\]
where $\alpha<\beta\in k\cup\{\pm\infty\}$ and $g_i,h_i$ are understood to be definable functions $R^{i-1}\to R$ (or possibly the ``constant functions" at $\pm \infty$) such that $g_i<h_i$ on the preceding product $(\alpha,\beta)\times\dots\times(g_{i-1},h_{i-1})$. Then the product above denotes the definable set of points \[\{(x_1,\dots,x_n)\in R^n: \alpha<x_i<\beta\land \bigwedge _{2\leq i\leq n}g_i(x_{<i})<x_i<h_i(x_{<i})\}\]

We will also use some abbreviated notation for other definable subsets of $R^n$. For example, $\{x_1>0\}$ will denote the set of points $x\in R^n$ such that $\pi_1(x)>0$.

\section{Constructing rigid fields}
A key fact for the main argument is that RCF (or any of its o-minimal expansions) has definable Skolem functions, which implies that for any subset $S$ of a model $N$, the definable closure of $S$ is an elementary submodel of $N$; this is clearly the smallest such submodel containing $S$. Given some elements $a_1,\dots,a_n\in R$ we use $k\langle a_1,\dots,a_n\rangle$ to denote this field, the real closure of $a_1,\dots,a_n$.
\\
\\
To produce a rigid non-Archimedean extension of finite transcendence degree over $k$, we employ exactly the same main argument as used in [1]. The extensions we construct will all have the form $k\langle a_1,\dots, a_n\rangle$ where $a_1>\alpha$ for all $\alpha\in k$ and all $a_i$ for $i\geq 2$ are in the convex hull of $k\langle a_1\rangle$ inside $R$. Given an automorphism $\sigma$ of such an extension $k\langle a_1,\dots, a_n\rangle$, we know that $\sigma(a_1),\dots,\sigma(a_n)\equiv a_1,\dots, a_n$. Also, since $k\langle a_1,\dots, a_n\rangle$ is the definable closure of $a_1,\dots,a_n$, there is some definable function $F_\sigma:R^n\to R^n$ such that $F_\sigma(a_1,\dots,a_n)=(\sigma(a_1),\dots,\sigma(a_n))$. If we wish to ensure that $k\langle a_1,\dots, a_n\rangle$ is rigid, it will suffice to ensure that every such definable map $F$ either fixes $(a_1,\dots,a_n)$ or else fails to preserve the type of this tuple. In the latter case, $F$ cannot be $F_\sigma$ for any automorphism $\sigma$. Hence, under our desired condition, for any automorphism $\sigma$, $F_\sigma$ must fix $(a_1,\dots,a_n)$ and so $\sigma$ must be the identity map.

To state the key lemma for this argument, we recall a definition from [1] (suitably adjusted for $n$ dimensions):
\begin{definition}
    An end-cell in $R^n$ is a definable set of the form 
    \[(\alpha,\infty)\times(g_2,h_2)\times\dots( g_n,h_n)\]
    where each $g_i,h_i$ is further assumed to be continuous on $(\alpha,\infty)\times\dots\times(g_{i-1},h_{i-1})$
\end{definition}

The main lemma we will need to complete this argument is the following generalization of Lemma 2.4 from [1]:

\begin{lemma}\label{MainLemma} Given a definable map $F:R^n\to R^n$ and an end cell $C$, there exists a possibly smaller end-cell $C'\subseteq C$ such that either $F\upharpoonright C'$ is the identity, or $F(C')\cap C'=\emptyset$. In fact, in the latter case we can ensure that the closure $\overline{C'}$ is disjoint from $F(C')$.
\end{lemma}

The additional claim about closure will not be needed until we reach the case of infinite transcendence degree. We will prove Lemma \ref{MainLemma} in the following section. Here we give the proof of the first main result using Lemma \ref{MainLemma}:

\begin{theorem}\label{FiniteThm}
    There exist non-Archimedean rigid real closed fields of all finite transcendence degrees $n\geq 2$ over $k$.
\end{theorem}
\begin{proof}
    Fix $n\geq 2$. Enumerate all definable maps $R^n\to R^n$ as $F_0,F_1,\dots$. Let $C_0:=(0,\infty)\times R^{n-1}$. Suppose $C_0\supseteq C_1\supseteq\dots\supseteq C_i$ have been defined such that $C_i$ is an end cell. Then let $C'_{i+1}$ be the end cell obtained by applying Lemma $\ref{MainLemma}$ to $C_i, F_i$ and take $C_{i+1}:=C'_{i+1}\cap \{i+1<x_1<\infty\}$. Now the intersection $\bigcap_{i<\omega} C_i$ may be viewed as a partial type $\tau(x_1,\dots,x_n)$ with the following properties: $\tau(x_1,\dots,x_n)\vdash x_1>\alpha$ for all $\alpha\in k$, and for any definable map $F:R^n\to R^n$, either $\tau(x_1,\dots,x_n)\vdash F(x_1,\dots,x_n)= (x_1,\dots,x_n)$ or else there is a formula $\varphi\in \tau$ such that $\varphi(x_1,\dots,x_n)\vdash \lnot\varphi(F(x_1,\dots,x_n))$. Namely, if $F=F_i$ under the chosen enumeration, then $\varphi$ may be taken as the definition of $C'_{i+1}$ in the above construction. We wish to take a realization $a_1,\dots,a_n$ of $\tau$ such that $a_1,\dots,a_n$ are algebraically independent. If this were not possible, then $\tau$ implies some algebraic relation among $a_1,\dots,a_n$ (in fact, one of the $a_i$ would be definable from the others). By compactness, some $C_i$ in the above construction implies this relation among $a_1,\dots,a_n$. This is not the case, however, as each $C_i$ has full dimension in $R^n$.
    
    Now consider the real closed field $K:=k\langle a_1,\dots,a_n\rangle$. This field is non-Archimedean as $a_1>\alpha$ for every $\alpha\in k$. Suppose $\sigma:K\to K$ is an automorphism. Then necessarily $\sigma(a_1),\dots,\sigma(a_n)\equiv a_1,\dots,a_n$. In particular, $\sigma(a_1),\dots,\sigma(a_n)\models \tau(x_1,\dots,x_n)$. Moreover, $\sigma(a_1),\dots,\sigma(a_n)\in k\langle a_1,\dots,a_n\rangle$ is exactly to say $(\sigma(a_1,\dots,\sigma(a_n))\in\text{dcl}(a_1,\dots,a_n)$, so there is a definable map $F:R^n\to R^n$ with $F(a_1,\dots,a_n)= (\sigma(a_1),\dots\sigma(a_n))$. Since $F(a_1,\dots,a_n)\models \tau(x_1,\dots,x_n)$, the construction of $\tau$ requires that $\tau(x_1,\dots,x_n)\vdash F(x_1,\dots,x_n)=(x_1,\dots,x_n)$. Then $\sigma$ is the identity on each $a_i$ and $K=\dcl(a_1,\dots,a_n)$ so $\sigma$ is the identity on $K$.

\end{proof}

As remarked in [1], one cannot directly employ the above argument to obtain a rigid non-Archimedean extension of degree $\omega$ by trying to diagonalize against all definable maps $R^\omega\to R^\omega$, as there are uncountably many such. This means that attempting to continue the above proof as a transfinite induction fails at the first limit step $\omega$: we can only possibly define $C_\omega$ to be some sub-cell of the intersection of all $C_i$ for $i<\omega$, but this intersection is not necessarily a definable set and so we have no guarantee of finding a sub-cell. It is essential that the collection we diagonalize against be countable. It will turn out that we can diagonalize against the countable collection of definable maps $R^n\to R^m$ for $n\geq m$.

Suppose we have an extension $K:=k\langle (a_i)_{i<\omega}\rangle$ and an automorphism $\sigma: K\to K$. For each $m<\omega$ there is $n\geq m$ such that $\sigma(a_1),\dots,\sigma(a_m)\in k\langle a_1,\dots,a_n\rangle$. As before, we know that $\sigma(a_1),\dots,\sigma(a_m)\equiv a_1,\dots,a_m$. Also as before, we have $(\sigma(a_1),\dots,\sigma(a_m))\in\dcl(a_1,\dots,a_n)$ so that there is a definable map $F:R^n\to R^m$ with $F(a_1,\dots,a_n)=(\sigma(a_1),\dots,\sigma(a_m))$. To show that $\sigma$ is the identity it suffices to show that $\sigma$ fixes $(a_1,\dots,a_m)$ for arbitrarily large $m$. Then we wish to choose $(a_i)_{i<\omega}$ to satisfy the following: for arbitrarily large $m$, for all $n\geq m$, given a definable map $F:R^n\to R^m$, either $F(a_1,\dots,a_n)=(a_1,\dots,a_m)$ or else $F(a_1,\dots,a_n)\not\equiv (a_1,\dots,a_m)$. This suggests the appropriate generalization of Lemma $\ref{MainLemma}$:

\begin{lemma}\label{ProjLemma}
    For any definable map $F:R^n\to R^m$ and any end cell $C\subseteq R^n$, there exists an end cell $C'\subseteq C$ such that either $F\upharpoonright C'=\pi_{\leq m}$, or else $\pi_{\leq m}(C')\cap F(C')=\emptyset$.
\end{lemma}

We will again postpone the proof to Section 3 and use Lemma \ref{ProjLemma} to prove the second main result; the particular construction implies Theorem \ref{FiniteThm} as well.

\begin{theorem}\label{InfiniteThm}
    There exists a non-Archimedean rigid real closed field of transcendence degree $\omega$ over $k$.
\end{theorem}
\begin{proof}
    Over all pairs $2\leq m\leq n$, enumerate all $k$-definable maps from $R^n\to R^m$, as $F_0,F_1,\dots$. We construct a partial type $\tau((x_i)_{i<\omega})$ in countably many variables as an intersection of a countable decreasing sequence of definable subsets $C_0\supseteq C_1\supseteq\dots$ of $R^\omega$. More specifically, each $C_i$, say defined by a formula $\varphi_i(x_1,\dots,x_{N_i})$, will be an end-cell when viewed as a definable subset of $R^{N_i}$. We will have $\varphi_{i+1}\to\varphi_{i}$ for all $i$ and we will take $\tau$ to be $\{\varphi_i(x_1,\dots,x_{N_i})\}_{i<\omega}$. Take $C_0$ to be the definable set $\{x_1>0\}$. Supposing $C_i$ has been defined and is an end cell in $R^{N_i
    }$, consider $F_i:R^n\to R^m$, $n\geq m\geq 2$. If $N_i<n$ then let $D$ be the $n$-dimensional end-cell $C_i\times R^{n-N_i}$. Now let $C'_{i+1}$ be the cell supplied by Lemma \ref{ProjLemma} applied to $D$ and $F_i$. Take $C_{i+1}$ to be $C'_{i+1}\cap \{x_1>i+1\}$. Otherwise if $N_i\geq n$ let $D$ be $\pi_{\leq n}(C_i)$ which is an end-cell of dimension $n$. Take $C'_{i+1}$ to be the end-cell supplied by Lemma \ref{ProjLemma} applied to $D$ and $F_i$. Then $C_i\cap \pi^{-1}_{\leq i}(C'_{i+1})$ is an end-cell contained in $C_i$. Take $C_{i+1}$ to be $C_i\cap \pi^{-1}_{\leq i}(C'_{i+1})\cap\{x_1>i+1\}$. (It is worth noting that the splitting into two cases is really non-essential here and is only mentioned for the sake of visualization; the two cases are identical if we think only of the defining formulas involved and the subsets of $R^\omega$ which they define). We may view $C_{i+1}$ as an end-cell in $R^{N_{i+1}}$ where $N_{i+1}:=\max\{n,N_i\}$. Now we have either that $F(x_1,\dots,x_n)=(x_1,\dots,x_m)$ for any $(x_1,\dots, x_{N_{i+1}})\in C_{i+1}$, or else for all $(x_1,\dots, x_{N_i+1})\in C_{i+1}$, there is no point $(y_1,\dots,y_{N_{i+1}})\in C_{i+1}$ with $(y_1,\dots,y_m)=F(x_1,\dots,x_n)$
    
    Now let $\tau((x_i)_{i<\omega})$ be the partial type in countable many variables given by $\bigcap_{i<\omega} C_i$. Then $\tau((x_i)_{i<\omega})\vdash x_1>\alpha$ for all $\alpha\in k$ and for any definable map $F:R^n\to R^m$, either $\tau((x_i)_{i<\omega})\vdash F(x_1,\dots,x_n)=(x_1,\dots,x_m)$ or else there is $\varphi(x_1,\dots,x_m)\in \tau((x_i)_{i<\omega})$ such that $\tau((x_i)_{i<\omega})\vdash \lnot\varphi(F(x_1,\dots,x_n))$. A little more precisely, there is a formula $\psi(x_1,\dots,x_n)\in\tau((x_i)_{i<\omega})$ such that $\psi(x_1,\dots,x_n)\vdash \lnot \exists y_{m+1}\dots y_n\psi(F(x_1),\dots,F(x_m),y_{m+1},\dots,y_n)$. In fact, if $F=F_i$, then $\psi$ may be taken as the formula defining $C_{i+1}'$ in the above construction.

    Since each $C_i$ is an end-cell, and in particular has full dimension as a subset of $R^{N_i}$, we may argue as in the finite case to take $(a_i)_{i<\omega}$ a realization in $R^\omega$ of $\tau(x_i)_{i<\omega}$ such that $(a_i)_{i<\omega}$ are algebraically independent. Now consider the real closure $K:=k\langle (a_i)_{i<\omega}\rangle$. This field is non-Archimedean as $a_1>\alpha$ for all $\alpha\in k$. Let $\sigma:K\to K$ be some automorphism of $K$. To show that $\sigma$ is the identity map, it suffices to show that $\sigma$ fixes the tuple $(a_1,\dots,a_m)$ for arbitrarily large $m$. In fact, we will show this is true for all $m\geq 2$.

    The tuple $(\sigma(a_1),\dots,\sigma(a_m))\in k\langle a_1,\dots,a_n\rangle$ for some sufficiently large $n\geq m$. Then there must be some definable map $F:R^n\to R^m$ with $F(a_1,\dots,a_n)=(\sigma(a_1),\dots,\sigma(a_m))$. Since $\sigma(a_1),\dots,\sigma(a_m)$ also satisfies the restriction of $\tau$ to the first $m$ variables, the construction of $\tau$ requires that $\tau((x_i)_{i<\omega})\vdash F(x_1,\dots,x_n)=(x_1,\dots,x_m)$. Hence $(\sigma(a_1),\dots,\sigma(a_m))=(a_1,\dots,a_m)$, which is what we needed to show.
\end{proof}

The fact that we obtained the desired conclusion for all $m\geq 2$ in the last paragraph of the proof means that each subfield $k\langle a_1,\dots,a_m\rangle$, $m\geq 2$, is itself rigid; hence this proof implies Theorem \ref{FiniteThm}. It also follows from the proof that the field $K$ has no nontrivial endomorphisms.

\section{Proving the key lemmas}
In [1], the main advantage of working in just 2 dimensions is that all boundary functions depend on a single variable. Accordingly, in the $n$-dimensional case we introduce the following stronger notion of end-cell:

\begin{definition}
    A \textit{one-parameter end cell} $C\subseteq R^n$ is an end cell all of whose boundary functions depend only on the variable $x_1$. That is, $C$ has the form:
    \[(\alpha,\infty)\times(g_2(x_1),h_2(x_1))\times\dots\times(g_n(x_1),h_n(x_1))\]
\end{definition}

Every open definable set unbounded above in the first coordinate (``unbounded on the right") contains an open cell unbounded on the right, by cell decomposition. Any open cell unbounded on the right contains an end-cell: this follows straightforwardly from piecewise continuity of definable functions. In order to make good use of 1-parameter end-cells, we need the following:
\begin{lemma}\label{ParamLemma}
Every open cell with first coordinate unbounded above contains a 1-parameter end cell.
\end{lemma}
\begin{proof}
    Let $C$ be the open cell given. After shrinking we are free to assume $C$ is an end-cell. We induct on the dimension of $C$. The claim is immediate when $n=\dim(C)=2$. By induction we may assume that $C$ has the form:
    \[(\alpha,\infty)\times (g_2(x_1),h_2(x_1))\times\dots\times(g_{n-1}(x_1), h_{n-1}(x_1))\times (g_n(x_{<n}), h_n(x_{<n}))\]
    By possibly shrinking to a smaller end-cell we may further assume each boundary function $g_i,h_i$ is $C^1$ on $\pi_{\leq i}(C)$.
    
    Let $i<n$. For brevity we will let $\pi$ denote $\pi_{\leq n-1}$. Our goal is to shrink and replace $C$ with a smaller end-cell such that $\pi(C)$ is still one-parameter, and for all $i<n$,  $0=\frac{\partial g_n}{\partial x_i}=\frac{\partial h_n}{\partial x_i}$ on all of $\pi(C)$. We do this by successively replacing the boundary functions $g_n,h_n$. We may handle $g_n$ and $h_n$ independently and the two cases are similar, so we will only present the method for replacing $g_n$.
    
    Fix some $i<n$ and suppose by induction that for all $2\leq j<i$, $0=\frac{\partial g_n}{\partial x_j}$ on all of $\pi(C)$. Definably decompose $\pi(C)$ according to the sign of $\frac{\partial g_n}{\partial x_i}$. One of the three sets in the decomposition must contain an open cell unbounded on the right and therefore by the inductive hypothesis contains a 1-parameter end cell. Hence, replacing $C$, we may assume that $\frac{\partial g_n}{\partial x_i}$ has the same sign on all of $\pi(C)$. If $\frac{\partial g_n}{\partial x_i}=0$ on all of $\pi(C)$ then we are already done. The cases $\frac{\partial g_n}{\partial x_i}>0$ and $\frac{\partial g_n}{\partial x_i}<0$ are similar, so we treat only the former. Choose some $\epsilon \in (0,1)\cap k$ and let $\Tilde{h}_i:=h_i+\epsilon(g_i-h_i)$. Let $\Tilde{g}_n(x_{<n}):=g_n(x_1,\dots,\tilde{h}_i(x_1),\dots,x_{n-1})$. Note that on $\pi(C)$, this function is $C^1$ and has $\frac{\partial \tilde{g}_n}{\partial x_i}=0$. We have also preserved $\frac{\partial \tilde{g}_n}{\partial x_j}=0$ for all $2\leq j<i$. By the assumption that $\frac{\partial g_n}{\partial x_i}>0$, we have that $\tilde{g}_n>g_n$ for all points of $\pi(C)$ satisfying $x_i<\tilde{h}_i(x_{1})$. 

    For any point of the form $(x_1,\dots,\tilde{h}_i(x_1),\dots,x_{n-1})\in\pi(C)$ we have
    \begin{flalign}
        \tilde{g}_n(x_1,\dots,\tilde{h}_i(x_1),\dots,x_{n-1}) &=g_n(x_1,\dots,\tilde{h}_i(x_1),\dots,x_{n-1})\nonumber\\
        &<h_n(x_1,\dots,\tilde{h}_i(x_1),\dots,x_{n-1})
    \end{flalign}
    
    Now $D_0:=\pi(C)\cap\{\tilde{g}_n(x_{<n})<h_n(x_{<n})\}$ is a definable subset of $\pi(C)$; moreover, continuity of $\tilde{g}_n,h_n$ on $\pi(C)$ imply that $D_0$ is an open subset of $\pi(C)$. $D_0$ also contains all points of $\pi(C)$ which satisfy $x_i=\tilde{h}_i(x_1)$ by (1). Then by openness, for every $a_1\in(\alpha,\infty)$, $D_0$ must also contain points with $x_1=a_1$ and $x_i<\tilde{h}_i(x_1)$. Then $D_1:=D_0\cap \{x_i<\tilde{h}_i(x_1)\}$ is again definable, open by continuity of $\tilde{h}_i$, and unbounded in the first coordinate by the previous sentence. Hence, by the inductive hypothesis, we may find $D$ a one-parameter end-cell inside $D_1$. Now $\tilde{g}_n<h_n$ on all of $D$, and since every point of $D$ satisfies $x_i<\tilde{h}_i(x_1)$, we have $\tilde{g}_n>g_n$ on all of $D$. Hence $D\times(\tilde{g}_n,h_n)\subseteq C$ and satisfies our requirements, completing the inductive step.
\end{proof}

With Lemma \ref{ParamLemma}, we can now prove Lemma \ref{MainLemma} in the same manner as in [1].
\\
\\
\textbf{Proof of Lemma \ref{MainLemma}}

    Let $C=(\alpha,\infty)\times (g_2,h_2)\times\dots\times(g_n,h_n)$ be an end-cell in $R^n$ and let definable $F:R^n\to R^n$ be given. After possibly shrinking $C$, we may assume that $F$ is continuous on all of $C$ and that $C$ is a 1-parameter end-cell.\\
    
    \textbf{Case 1:} If there is an end cell $C'\subseteq C$ such that $F\upharpoonright C'$ is the identity, then we are done.
    \\
    \\
    If we are not in Case 1, then definably decompose $C$ into the set of points fixed by $F$ and those not fixed by $F$. By assumption, the former contains no end-cell, so the latter must. Hence we may assume that $F(P)\neq P$ for all $P\in C$.

    Let $f: (\alpha,\infty)\to R^{n-1}$ be any definable function continuous on $(\alpha,\infty)$ such that $(x,f(x))\in C$ for all $x\in(\alpha,\infty)$. Let $\mu_f:R\to R$ be the definable function $\mu_f(x):=\pi_1(F(x,f(x))$ and let $\nu_f:R\to R^{n-1}$ be the definable function $\nu_f(x):=(\pi_2(F(x,f(x)),\dots,\pi_n(F(x,f(x)))$.

    \textbf{Case 2:} Suppose there is some definable continuous $f:(\alpha,\infty)\to R^{n-1}$ whose graph $\Gamma(f)\subseteq C$ and $\mu_f(x)\not\to \infty$ as $x\to \infty$. By piecewise monotonicity, there is $\alpha'>\alpha$ and $\beta$ such that $x>\alpha'\implies \mu_f(x)<\beta$. Moreover, we are free to assume that $\alpha'>\beta$. By continuity, $D:=(\pi_1\circ F)^{-1}(-\infty,\beta)\cap \{x_1>\alpha'\}\cap C$ is an open definable subset of $C$. By assumption, $D$ contains $\Gamma(f\upharpoonright(\alpha',\infty))$ so $D$ is unbounded on the right. Then $D$ contains an end-cell $C'$. Every point of $\overline {C'}$ has $x_1\geq \alpha'$ while every point of $F(C')$ has $x_1<\beta$ so that $\overline{C'}\cap F(C')=\emptyset$\\

    Now if we are not in Case 2, we may assume that every definable continuous $f:(\alpha,\infty)\to R^{n-1}$ with $\Gamma(f)\subseteq C$ has $\mu_f(x)\to \infty$ as $x\to \infty$. Given such an $f$, by piecewise monotonicity, after possibly shrinking $C$, we may assume that $\mu_f$ is strictly increasing. Then for $x>\max\{\mu_f(\alpha),\alpha\}$ we define $f^*(x):=\nu_f(\mu_f^{-1}(x))$. Then $\Gamma(f^*)$ is (possibly a truncation of) $F(\Gamma(f))$.

    \textbf{Case 3:} Suppose there is some $f$ as above such that $f(x)\neq f^*(x)$ for large enough $x$. By piecewise monotonicity, there must be some coordinate $2\leq i\leq n$ such that $f_i(x)\neq f^*_i(x)$ for sufficiently large $x$. After possibly increasing $\alpha$ and shrinking $C$ we are free to assume that either $f_i>f^*_i$ or $f^*_i<f_i$ on their whole common domain (say $x>\beta)$. We will treat just the first case as the second is similar.
    
    Let $\phi_0,\phi_1$ be definable continuous functions $R\to R$ with $\phi_0(x)<f^*_i(x)<\phi_1(x)<f_i(x)$ for all $x>\beta$. By possibly shrinking $C$ towards the right and possibly replacing $g_i$ with $\phi_1+\epsilon(f_i-\phi_1)$ for some $\epsilon\in(0,1)\cap k$, we may assume that $\phi_1(x)<g_i(x)$ for all $x>\beta$. Let $D$ be the definable open set $\{\phi_0(x_1)<x_i<\phi_1(x_1)\}$. Then since $F$ is continuous on $C$ and $F(\Gamma(f))\subseteq D$, the definable set $A:=F^{-1}(D)\cap C$ is open and contains all of $\Gamma(f)$. Then $A$ is also unbounded on the right, so $A$ contains some 1-parameter end cell $C'=(\alpha',\infty)\times(\tilde{g}_2,\tilde{h}_2)\times\dots\times(\tilde{g}_n,\tilde{h}_n)$. We have ensured $\phi_1(x)<g_i(x)\leq\tilde{g}_i(x)$ for all $x\in(\alpha',\infty)$ and $F(C')\subseteq D$. Every point in $\overline{C'}$ satisfies $x_i \geq \tilde{g}_i(x_1)>\phi_1(x_1)$ so that $\overline{C'}\cap D=\emptyset$ which implies $\overline{C'}\cap F(C')=\emptyset$.
    \\
    
    Now suppose towards contradiction that we are in none of the above cases. For each $r\in[0,1]\cap k$ consider the definable function
    \[f_r(x):= (g_2(x)+r(h_2(x)-g_2(x)),\frac{g_3(x)+h_3(x)}{2},\dots,\frac{g_n(x)+h_n(x)}{2})\]
    As we are not in Case 3, for each $r\in k\cap[0,1]$ there is $m_r$ such that $x>m_r$ implies $f_r^*(x)=f_r(x)$. Taking infima makes the assignment $r\to m_r$ definable. Then there is some subinterval $[a,b]\subseteq [0,1]\cap k$ with $a<b$ on which $m_r$ is a continuous, monotone function of $r$. Shrinking $C$ by replacing $g_2$ with $g_2+a(h_2-g_2)$ and $h_2$ with $g_2+b(h_2-g_2)$, we may assume that $m_r$ is continuous and monotone on the whole interval $[0,1]\cap k$. Then by considering either $m_0$ or $m_1$ we obtain $m$ such that $x>m$ implies $f^*_r(x)=f_r(x)$ for all $r$. Shrinking $C$ again by replacing $\alpha$ with $\max\{\alpha,m\}$ we may now assume that $F$ maps $\Gamma(f_r)$ into itself for all $r$.

    Now let $\psi:(\alpha,\infty)\to(0,1)$ be an increasing definable continuous bijection. Define $f(x):=f_{\psi(x)}(x)$. For any $r$, if $f_r(x)=f(x)$, then as the graphs of $f_r,f_{r'}$ are disjoint for distinct $r,r'$, it must be that $r=\psi(x)$. This is to say that each $\Gamma(f_r)$ intersects $\Gamma(f)$ at a unique point, namely where $x=\psi^{-1}(r)$. As we are not in Case 3, after restricting attention to sufficiently large $x$, we may assume $F$ maps $\Gamma(f)$ into itself. For such $x$, the point $(x,f(x))$ is on both $\Gamma(f)$ and $\Gamma(f_{\psi(x)})$. Moreover, as $F$ maps each of these graphs into itself, $F(x,f(x))$ must be another intersection point of these graphs. But $(x,f(x))$ is the unique intersection of $\Gamma(f)$ and $\Gamma(f_{\psi(x)})$. Hence $(x,f(x))$ is fixed by $F$, contrary to the assumption that we are not in Case 1. \qed
    \\

And now we prove the second key lemma:
\paragraph{Proof of Lemma \ref{ProjLemma}}
    Let an end-cell $C\subseteq R^n$ and definable function $F:R^n\to R^m$ be given. Let $\pi$ denote $\pi_{\leq m}$. After shrinking we may suppose $C$ is 1-parameter. Definably decompose $C$ into the set of points where $F=\pi$ and where $F\neq \pi$. If the former definable subset contains an end-cell, we are done. Otherwise, the latter set contains an end-cell and so we may proceed assuming that for all points $P\in C$, $F(P)\neq \pi(P)$. After possibly shrinking further, we may also assume that $F$ is continuous on all $C$.

    Note that $\pi(C)$ is also a 1-parameter end cell of dimension $m$. Say $C=(\alpha,\infty)\times (g_2,h_2)\times\dots\times(g_n,h_n)$. Let $G:R^m\to R^n$ be the definable map \[G:(x_1,\dots,x_m)\mapsto (x_1, \dots,x_m,\frac{g_{m+1}(x_1)+h_{m+1}(x_1)}{2},\dots,\frac{g_n(x_1)+h_n(x_1)}{2})\]
    Let $\tilde{F}:R^m\to R^m$ be the definable map $F\circ G$. By the assumption that $F(P)\neq \pi(P)$ for all $P\in C$, we have $\tilde{F}(Q)\neq Q$ for all $Q\in \pi(C)$. Then applying Lemma $\ref{MainLemma}$ to $\pi(C)$ we obtain an end-cell $D\subseteq \pi(C)$ such that $\overline{D}\cap \tilde{F}(D)=\emptyset$. By the assumption that $F$ is continuous on all $C$, the set $U:=F^{-1}(R^m\setminus \overline{D})\cap C$ is an open definable subset of $C$, and this set contains all of $G(D)$. Likewise, continuity of projection implies that $V:=\pi^{-1}(D)\cap C$ is an open definable subset of $C$ containing all of $G(D)$. Then $U\cap V$ is an open definable subset of $C$ containing all of $G(D)$; in particular it is unbounded on the right. The choice of $U,V$ implies that $F(U\cap V)\cap \pi(U\cap V)=\emptyset$. Take some end-cell contained inside $U\cap V$ to conclude the claim.\qed

\section{References}
[1] David Marker and Charles Steinhorn. Rigid Real Closed Fields. Preprint. https://arxiv.org/abs/2407.00542.
\\

\noindent[2] 
  Lou van den Dries. \textit{Tame Topology and O-Minimal Structures}. Cambridge University Press, 1998.

\end{document}